\documentclass[11pt]{amsart}

\usepackage{graphicx,color}
\usepackage{amsmath}
\usepackage{amsthm}
\usepackage{amssymb}
\usepackage[all]{xy}

\usepackage{geometry}
\usepackage[T1]{fontenc}
\usepackage{textcomp}
\usepackage{times}
\usepackage[scaled=0.92]{helvet}
\usepackage{calrsfs}

\newtheorem{introtheorem}{Theorem}

\newtheorem{introcorollary}[introtheorem]{Corollary}
\theoremstyle{definition}
\newtheorem{introexample}[introtheorem]{Example}
\newtheorem{introdefinition}[introtheorem]{Definition}
\theoremstyle{plain}
\swapnumbers
\newtheorem{theorem}{Theorem}[section]
\newtheorem{proposition}[theorem]{Proposition}
\newtheorem{lemma}[theorem]{Lemma}
\newtheorem{corollary}[theorem]{Corollary}
\theoremstyle{definition}
\newtheorem{definition}[theorem]{Definition}
\newtheorem*{definition*}{Definition}
\newtheorem{example}[theorem]{Example}
\newtheorem{examples}[theorem]{Examples}
\newtheorem*{example*}{Example}
\newtheorem{remark}[theorem]{Remark}
\newtheorem*{remark*}{Remark}

\newtheorem{nd}[theorem]{Notations/Definitions}

\renewcommand{\bar}{\overline}
\newcommand{\isomto}{\overset{\sim}{\rightarrow}}
\DeclareMathOperator{\Hom}{Hom}
\DeclareMathOperator{\GL}{GL}
\newcommand{\A}{\mathbf{A}}
\newcommand{\C}{\mathbf{C}}
\newcommand{\D}{\mathbf{D}}
\newcommand{\G}{\mathbf{G}}

\newcommand{\Q}{\mathbf{Q}}
\newcommand{\R}{\mathbf{R}}
\newcommand{\T}{\mathbf{T}}

\newcommand{\Z}{\mathbf{Z}}
\newcommand{\cH}{\mathcal{H}}

\newcommand{\cM}{\mathcal{M}}
\newcommand{\cO}{\mathcal{O}}
\newcommand{\m}{\mathfrak{m}}
\newcommand{\n}{\mathfrak{n}}
\newcommand{\p}{\mathfrak{p}}

\newcommand{\ab}{\mathrm{ab}}

\usepackage{amsmath,amssymb,amsfonts,amsthm}

\begin{document}

\date{\today\ (version 1.0)}
\title[Hecke algebra isomorphisms and adelic points on algebraic groups]{Hecke algebra isomorphisms and \\ adelic points on algebraic groups}
\author[G.~Cornelissen]{Gunther Cornelissen}
\address{\normalfont Mathematisch Instituut, Universiteit Utrecht, Postbus 80.010, 3508 TA Utrecht, Nederland}
\email{g.cornelissen@uu.nl, v.z.karemaker@uu.nl}
\author[V.~Karemaker]{Valentijn Karemaker}
\subjclass[2010]{11F70, 11R56, 14L10, 20C08, 20G35, 22D20}
\keywords{algebraic groups, adeles, Hecke algebras, arithmetic equivalence}
\thanks{Part of this work was done when the authors visited the University of Warwick; we are grateful to Richard Sharp for hosting us. The second author would like to thank Jan Nekov\'a\v{r} for interesting discussions. We thank Manfred Lochter for sending us a copy of \cite{Lochter}, Dipendra Prasad for pointing us to existing results on isomorphisms of point groups over rings, and Wilberd van der Kallen and Maarten Solleveld for helpful comments on an earlier version of this paper.}

\begin{abstract}
Let $G$ denote a linear algebraic group over $\Q$ and $K$ and $L$ two number fields. Assume that there is a group isomorphism $G(\A_{K,f}) \cong G(\A_{L,f})$ of points on $G$ over the finite adeles of $K$ and $L$, respectively. We establish conditions on the group $G$, related to the structure of its Borel groups, under which  $K$ and $L$ have isomorphic adele rings. Under these conditions, if $K$ or $L$ is a \emph{Galois} extension of $\Q$ and $G(\A_{K,f}) \cong G(\A_{L,f})$, then $K$ and $L$ are isomorphic as fields. 

We use this result to show that if for two number fields $K$ and $L$ that are Galois over $\Q$, the finite Hecke algebras for $\GL(n)$ (for fixed $n \geq 2$) are isomorphic by an isometry for the $L^1$-norm, then the fields $K$ and $L$ are isomorphic. This can be viewed as an analogue in the theory of automorphic representations of the theorem of Neukirch that the absolute Galois group of a number field determines the field, if it is Galois over $\Q$.
\end{abstract}

\maketitle

\section{Introduction}
Suppose that $G$ is a linear algebraic group over $\Q$, and $K$ and $L$ are two number fields such that the (finite adelic) Hecke algebras for $G$ over $K$ and $L$ are isomorphic.
As we will see, a closely related hypothesis is: suppose that the groups of finite adelic points on $G$ are isomorphic for $K$ and $L$. What does this imply about the fields $K$ and $L$? 
Before stating the general result, let us discuss an example. 

\begin{introexample} \label{e1} If $G=\G_a^r \times \G_m^s$ for any integers $r,s$, then for any two distinct imaginary quadratic fields $K$ and $L$ of discriminant $<-8$ we have an isomorphism of topological groups $G(\A_{K,f}) \cong G(\A_{L,f})$ while $\A_K \not\cong \A_L$  and  $\A_{K,f} \not\cong \A_{L,f}$ (cf.\ Section \ref{counterex}). To prove this, one determines separately the abstract structures of the additive and multiplicative groups of the adele ring $\A_K$ and sees that they depend on only a few arithmetical invariants, allowing for a lot of freedom in ``exchanging local factors''. This example illustrates that at least some condition on the rank, unipotent rank, and action of the torus on the unipotent part will be required to deduce that we have a ring isomorphism $\A_K \cong \A_L$. 
\end{introexample} 

Let us now state the main technical condition, which we will elaborate on in Section \ref{nice}. 
\begin{introdefinition}
Let $G$ denote a linear algebraic group over $\Q$. We call $G$ \emph{fertile} for a field $K/\Q$ if $G$ contains a Borel group $B$ which is split over $K$ as $B = T \ltimes U$, such that over $K$, the split maximal torus $T \neq \{1\}$ acts non-trivially by conjugation on the abelianisation of the maximal unipotent group $U\neq \{0\}$.
 We will prove in Proposition \ref{Wilb} that $G$ being fertile for $K$ is equivalent to $G$ being a $K$-split group whose connected component is \emph{not} a direct product of a torus and a unipotent group. 
\end{introdefinition}

Split tori and unipotent groups are not fertile for any $K$. On the other hand, for $n \geq 2$, $\GL(n)$ is fertile for all $K$. In general, fertility is slightly stronger than non-commutativity. Roughly speaking, it says that the group has semisimple elements that do not commute with some unipotent elements. 

\medskip

Our first main result is: 

\begin{introtheorem} \label{mainG}
Let $K$ and $L$ be two number fields, and let $G$ denote a 
linear algebraic group over~$\Q$ which is fertile for $K$ and $L$. There is a topological group isomorphism of finite adelic point groups $G(\A_{K,f})~\cong~G(\A_{L,f})$ if and only if there is a  topological ring isomorphism $\A_K \cong \A_L$. 
\end{introtheorem}

An isomorphism of adele rings $\A_K \cong \A_L$ implies (but is generally stronger than) arithmetic equivalence of $K$ and $L$ (Komatsu \cite{Komatsu}, cf.\ \cite{MR1638821}, VI.2). Recall  that $K$ and $L$ are said to be arithmetically equivalent if they have the same Dedekind zeta function: $\zeta_K=\zeta_L$ . If $K$ or $L$ is a Galois extension of $\Q$, then this is known to imply that $K$ and $L$ are isomorphic as fields (Ga{\ss}mann \cite{Gassmann}, cf.\ \cite{MR1638821}, III.1).

The question whether $G(R) \cong G(S)$ for algebraic groups $G$ and rings $R, S$ implies a ring isomorphism $R \cong S$ has been considered before (following seminal work of van der Waerden and Schreier from 1928 \cite{SV}), most notably when $G=\GL_n$ for $n \geq 3$ or when $G$ is a Chevalley group and $R$ and $S$ are integral domains (see, e.g., \cite{Chen}, \cite{Petechuk} and the references therein). The methods employed there make extensive use of root data and Lie algebras.

By contrast, our proof of Theorem \ref{mainG} uses number theory in adele rings and, by not passing to Lie algebras, applies to a more general class of (not necessarily reductive) algebraic groups. First, we prove in general that maximal divisible subgroups $\mathbf D$ of $G(\A_{K,f})$ and maximal unipotent point groups are the same up to conjugacy (Proposition \ref{Urecovery}; note that this does not apply at the archimedean places). The torus $\mathbf T$ (as a quotient of the normaliser $\mathbf N$ of the unipotent point group $\mathbf D$ by itself) acts on the abelian group $\mathbf V = [\mathbf N, \mathbf D]/[\mathbf D, \mathbf D],$ that decomposes as a sum of one-dimensional $\mathbf T$-modules, on which $\mathbf T$ acts by multiplication with powers. Now we use a formula of Siegel, which allows us to express any adele as a linear combination of fixed powers, to show how this implies that the centre of the endomorphism ring of the $\mathbf T$-module $\mathbf V$ is a a cartesian power of the finite adele ring. We then use the structure of the maximal principal ideals in the finite adele ring to find from these data the adele ring itself. 

\begin{introexample}Consider $G=\GL(2)$. Then $\mathbf D = \left ( \begin{smallmatrix} 1 & \A_{K,f} \\ 0 & 1 \end{smallmatrix} \right) \cong (\A_{K,f},+)$ is (conjugate to) a group of strictly upper triangular matrices, $\mathbf N =  \left ( \begin{smallmatrix} \A^*_{K,f} & \A_{K,f} \\ 0 & \A_{K,f}^* \end{smallmatrix} \right)$ and  $\T \cong (\A_{K,f}^*,\cdot)^2$ (represented as diagonal matrices) acts on $\mathbf V \cong \mathbf D$, represented as upper triangular matrices, by multiplication. Now $\mathrm{End}_{\mathbf T} \mathbf V \cong \A_{K,f}$ as (topological) rings. 
\end{introexample} 

\medskip

By the (finite adelic) \emph{Hecke algebra} for $G$ over $K$, we mean the convolution algebra $\cH_G(K):=C_c^\infty(G(\A_{K,f}),\R)$ of locally constant compactly supported real-valued functions on $G(\A_{K,f})$.
By an \emph{$L^1$-isomorphism} of Hecke algebra we mean one that respects the $L^1$-norm. The second main result is the following: 

\begin{introtheorem} \label{mainH}
Let $K$ and $L$ be two number fields, and let $G$ denote a linear algebraic group over $\Q$ that is fertile for $K$ and $L$. There is 
an $L^1$-isomorphism of Hecke algebras $\cH_G(K) \cong \cH_G(L)$ if and only if there is a ring isomorphism $\A_K \cong \A_L$. 
\end{introtheorem}

This follows from the previous theorem by using some density results in functional analysis and a theorem on the reconstruction of an isomorphism of groups from an isometry of $L^1$-group algebras due to Kawada \cite{Kawada} and Wendel  \cite{MR0049910}. It seems that the analytic condition of being an isometry for the $L^1$-norm is necessary (forthcoming work of the second author uses the Bernstein decomposition to show that purely algebraic isomorphisms of Hecke algebras of local fields cannot distinguish local fields). 

\medskip

Let $G_K$ denote the absolute Galois group of a number field $K$ that is Galois over $\Q$. Neukirch \cite{Neukirch} has proven that $G_K$ determines $K$ (Uchida \cite{Uchida} later removed the condition that $K$ is Galois over $\Q$). The set of one-dimensional representations of $G_K$, i.e., the abelianisation $G_K^{\mathrm{{\tiny ab}}}$, far from determines $K$ (compare \cite{Onabe} or \cite{AngSte}). Several years ago, in connection with the results in \cite{CM}, Jonathan Rosenberg asked the first author whether, in a suitable sense, two-dimensional irreducible---the ``lowest-dimensional non-abelian''--- representations of $G_K$ determine $K$. By the philosophy of the global Langlands programme, such representations of $G_K$ in $\GL(n,\C)$ should give rise to automorphic representations, i.e., to certain modules over the Hecke algebra $\cH_{\GL(n)}(K)$. If we consider the analogue of this question in the setting of $\cH_{\GL(n)}(K)$-modules instead of $n$-dimensional Galois representations, our main theorem implies a kind of ``automorphic anabelian theorem'':

\begin{introcorollary} \label{corc}
Suppose that $K$ and $L$ are number fields that are Galois over $\Q$. There is
an $L^1$-algebra isomorphism of Hecke algebras $\cH_{\GL(n)}(K) \cong \cH_{\GL(n)}(L)$ for some $n \geq 2$ if and only if there is a field isomorphism $K \cong L$. 
\end{introcorollary}
\medskip 

The paper has the following structure. In Section \ref{counterex}, we discuss what happens if $G$ is the additive or multiplicative group or a direct product thereof. In Section \ref{nice}, we introduce and discuss the notion of fertility. In Section \ref{div}, we prove that maximal divisibility is equivalent to unipotency in finite-adelic point groups. In Section \ref{red3} 
 we use this to prove Theorem \ref{mainG}. In Section \ref{HH}, we discuss Hecke algebras and prove Theorem \ref{mainH}. At the end of the paper, we discuss some open problems. 

\section{Additive and multiplicative groups of adeles} \label{counterex}

In this section, we elaborate on Example \ref{e1} from the introduction. We discuss the group structure of the additive and multiplicative groups of adeles of a number field, and we recall the notions of local isomorphism of number fields and its relation to isomorphism of adele rings and arithmetic equivalence. We introduce \emph{local} additive and multiplicative isomorphisms and prove that their existence implies arithmetic equivalence. 
\subsection*{Arithmetic equivalence and local isomorphism}

\begin{nd} \label{deflocal}
If $K$ is a number field with ring of integers $\cO_K$, let $M_K$ denote the set of all places of $K$, $M_{K,f}$ the set of non-archimedean places of $K$, and $M_{K,\infty}$ the set of archimedean places. If $\p \in M_{K,f}$ is a prime ideal, then $K_{\p}$ denotes the completion of $K$ at $\p$, and $\cO_{K,\p}$ its ring of integers. Let $e(\p)$ and $f(\p)$ denote the ramification and residue degrees of $\p$ over the rational prime $p$ below $\p$, respectively. The \emph{decomposition type} of a rational prime $p$ in a field $K$ is the sequence $(f(\p))_{\p \mid p}$ of residue degrees of the prime ideals of $K$ above $p$, in increasing order, with multiplicities.

We use the notation $\prod\nolimits^{\prime} (G_i,H_i)$ for the restricted product of the group $G_i$ with respect to\ the subgroups $H_i$. We denote by $$\A_K={\prod\limits_{\p \in M_{K}}}^{\hspace*{-2mm} \prime}\left(K_{\p},\cO_{K,\p}\right)$$ the adele ring of $K$, and by $$\A_{K,f}={\prod\limits_{\p \in M_{K,f}}}^{\hspace*{-3mm} \prime}\left(K_{\p},\cO_{K,\p}\right)$$ its ring of finite adeles. 

Two number fields $K$ and $L$ are \emph{arithmetically equivalent} if for all but finitely many prime numbers $p$, the decomposition types of $p$ in $K$ and $L$ coincide.  

Two number fields $K$ and $L$ are called \emph{locally isomorphic} if there is a bijection $  \varphi \colon M_{K,f} \rightarrow M_{L,f}$ between their sets of prime ideals such that the corresponding local fields are topologically isomorphic, i.e.\ $K_{\p} \cong L_{\varphi(\p)}$ for all $\p \in M_{K,f}$. 
\end{nd}

 The main properties are summarised in the following proposition (see e.g.\ \cite{MR1638821}, III.1 and VI.2): 

\begin{proposition}\label{Klingen} Let $K$ and $L$ be number fields. Then:
\begin{itemize}
 \item[(i)]
$K$ and $L$ are locally isomorphic if and only if the adele rings $\A_K$ and $\A_L$ are isomorphic as topological rings, if and only if 
the rings of finite adeles $\A_{K,f}$ and $\A_{L,f}$ are isomorphic as topological rings.
 \item[(ii)] $K$ and $L$ are arithmetically equivalent if and only if $\zeta_K = \zeta_L$, if and only if 
there is a bijection $  \varphi \colon M_{K,f} \rightarrow M_{L,f}$ such that the local fields $K_{\p} \cong L_{\varphi(\p)}$ are isomorphic for \emph{all but finitely many} $\p \in M_{K,f}$. 
 \item[(iii)] We have $ K \cong L \Rightarrow \A_K \cong \A_L$ (as topological rings) $\Rightarrow \zeta_K = \zeta_L $ and none of the implications can be reversed in general, but if $K$ or $L$ is Galois over $\Q$, then all implications can be reversed. $\qed$
\end{itemize}
\end{proposition}

\subsection*{The additive group of adeles}

\begin{proposition} \label{a}
If $H$ is a number field, then there are topological isomorphisms of additive groups
\[
(\A_{H,f},+) \cong (\A_{\Q,f}^{[H: \Q]},+)
\]
and $$(\A_H,+) \cong (\A_{\Q}^{[H:\Q]},+).$$  \end{proposition}

\begin{proof}
If $\p$ is a prime of $H$ above the rational prime $p$, then $\mathcal{O}_{H,\p}$ is a free $\Z_p$-module of rank $e(\p)f(\p)$ (cf.\ \cite{cassels} 5.3-5.4), and tensoring with $\Q$ gives a compatible diagram of isomorphisms of additive groups 
$$\xymatrix{ (\mathcal{O}_{H,\p},+) \ar@{->}[r]_-{\sim} \ar@{^{(}->}[d]  & (\Z_p^{e(\p)f(\p)},+)  \ar@{^{(}->}[d] \\ 
(H_{\p},+) \ar@{->}[r]_-{\sim} & (\Q_p^{e(\p)f(\p)},+)
     }$$
which we can sum over all $\p \mid p$ for fixed $p$, to find 
$$\xymatrix{ (\bigoplus\limits_{\p \mid p} \mathcal{O}_{H,\p},+) \ar@{->}[r]_-{\sim} \ar@{^{(}->}[d]  & (\Z_p^n,+)  \ar@{^{(}->}[d] \\ 
(\bigoplus\limits_{\p \mid p} H_{\p},+) \ar@{->}[r]_-{\sim} & (\Q_p^n,+)
     }$$
for $n=[H:\Q]$. It follows that 
\begin{align*} (\A_{H,f},+) &= {\prod\limits_{p \in M_{\Q,f}}}^{\prime} (\bigoplus_{\p|p}(H_{\p},+),\bigoplus_{\p|p}(\cO_{H,\p},+)) \\ &\cong {\prod\limits_{p \in M_{\Q,f}}}^{\hspace*{0mm} \prime} ((\Q_p^n,+),(\Z_p^n,+)) \\ & \cong (\A_{\Q,f}^n,+) 
\end{align*}
and hence
\begin{align*} (\A_H,+) &= (\A_{H,f},+) \times (\R^n,+) \\ & \cong (\A_{\Q,f}^n,+) \times (\R^n,+) \cong (\A_{\Q}^n,+). 
\end{align*}
\end{proof}
\begin{corollary}
The additive groups $(\A_K,+)$ and $(\A_L,+)$ are isomorphic (as topological groups) for two number fields $K$ and $L$ if and only if $K$ and $L$ have have the same degree over $\Q$.  For finite adeles, $[K \colon \Q] = [L \colon \Q]$ implies $(\A_{K,f},+) \cong (\A_{L,f},+)$.
\end{corollary}
\begin{proof}
By Proposition \ref{a}, we know that $[K:\Q] = [L:\Q]$ implies that $(\A_{K,f},+) \cong (\A_{L,f},+)$ and $(\A_K,+) \cong (\A_L,+)$.

Conversely, a topological isomorphism $(\A_K,+) \cong (\A_L,+)$ of additive groups induces a homeomorphism between their respective connected components of the identity, i.e.,
\[
\R^{[K:\Q]} \cong \R^{[L:\Q]}.
\]
From this, using the fact that if $m\neq n$, then $\R^m-\{0\}$ and $\R^m-\{0\}$ have different homology groups, 
we deduce that $[K:\Q] = [L:\Q]$. 
\end{proof} 

If additionally, the isomorphism is ``local'', i.e.\ induced by local additive isomorphisms, then we have the following result: 

\begin{proposition}
Let $K$ and $L$ be number fields such that there is a bijection $\varphi \colon M_{K,f} \rightarrow M_{L,f}$ with, for almost all places $\p$, isomorphisms of topological groups $\Phi_{\p} \colon (K_{\p},+) \cong (L_{\varphi(\p)},+)$. Then $K$ and $L$ are arithmetically equivalent.
\end{proposition}

\begin{proof}
We may view each $K_{\p}$ as a $[K_{\p}:\Q_p] =e(\p)f(\p)$-dimensional topological $\Q_p$-vector space, where $p \in \Q$ is the prime lying below $\p$. Similarly, for  $q \in \Q$ the prime lying below $\varphi(\p) \in L_{\varphi(\p)}$, we find that $L_{\varphi(\p)}$ is an $[L_{\varphi(\p)}:\Q_q] =e(\varphi(\p))f(\varphi(\p))$-dimensional topological $\Q_q$-vector space. We will write $n = e(\p)f(\p)$ and $m = e(\varphi(\p))f(\varphi(\p))$. Thus, we have an isomorphism of topological groups
\[
\Phi_{\p} \colon (\Q_p^{n},+) \isomto (\Q_{q}^{m},+).
\]
which must map $\Z^n$ injectively and discretely onto a subgroup of $\Q^m_q$ of the form $\bigoplus_{i=1}^n \nu_i \Z$, where the $\nu_i$ are $\Z$-linearly independent.

We indicate topological closure by a bar. Since the group of integers $\Z$ is dense in both $\Z_p$ and $\Z_q$ and $\Phi_{\p}$ is a homeomorphism, we have 
\[
\Phi_{\p}(\Z_p^n) = \Phi_{\p}(\overline{\Z}^n) = \overline{\Phi_{\p}(\Z^n)} = 
\overline{\bigoplus_{i=1}^n \nu_i \Z} = \sum_{i=1}^n \nu_i \Z_q \cong \Z_q^{n'},
\]
where $n' \leq n$. In the last step, we have used that since $\Z_{p}$ is a principal ideal domain, any submodule of the free module $\Z_{p}^n$ is free. Thus, $\Phi_{\p}$ restricts to a group isomorphism 
\[
\Phi'_{\p} \colon (\Z_p^{n},+) \isomto (\Z_{q}^{n'},+).
\]
We know (cf. Lemma 52.6 of \cite{classicalfields}) that the only $p$-divisible subgroup of $(\Z_p^n,+)$ is $\{ 0 \}$, whereas for $q \neq p$, every element of $(\Z_p^n,+)$ is $q$-divisible. This group theoretic property ensures that $p = q$, that is, 
\[
\Phi'_{\p} \colon (\Z_p^{n},+) \isomto (\Z_{p}^{n'},+) \subset (\Q_p^m,+).
\]
As a group homomorphism, $\Phi'_{\p}$ preserves the subgroup $p \Z_p^n$, and by considering the quotient, we find an isomorphism $(\mathbf{F}_p^n,+) \cong (\mathbf{F}_p^{n'},+)$, which, by counting elements, implies that $n=n'$.  Moreover, we see that $n=n' \leq m$, and since the isomorphism is invertible, we also obtain $m \leq n$, hence $n = m$. 

We conclude that for all non-archimedean places $\p \in M_{K,f}$, we must have $$e(\p)f(\p) = e(\varphi(\p))f(\varphi(\p)).$$ At all but finitely many primes $p$, both $K$ and $L$ are unramified, so the local maps $\Phi_{\p}$ will ensure that $f(\p) = f(\varphi(\p))$ for all but finitely many residue field degrees $f(\p)$. By Proposition \ref{Klingen}, this implies that $K$ and $L$ are arithmetically equivalent.
\end{proof} 

\begin{remark} \label{remV} If all but finitely many residue field degrees of $K$ and $L$ match, then in fact \emph{all} residue field degrees match, by a result of Perlis (the equivalence of (b) and (c) in Theorem 1 in  \cite{MR0447188}). This in turn implies that all ramification degrees match. So whereas two arithmetically equivalent number fields may have different ramification degrees at finitely many places, the above isomorphism excludes this possibility. However, this is still weaker than local isomorphism, since the ramification degree does not uniquely determine the ramified part of a local field extension.
\end{remark}

\subsection*{The multiplicative group of adeles} 

\begin{proposition}\label{mult} If $H$ is a number field with $r_1$ real and $r_2$ complex places, then there is a topological group isomorphism
\[ (\A_H^*,\cdot) \cong (\R^*)^{r_1} \times (\C^*)^{r_2} \times \left( \bigoplus_{\Z} \Z\right) \times \hat{\Z}^{[H:\Q]} \times \prod_{\p \in M_{H,f}} (\bar H_{\p}^* \times \mu_{p^\infty}(H_{\p}))  \]
where $\bar H_{\p}^*$ is the multiplicative group of the residue field of $H$ at $\p$ (a cyclic group of order $p^{f(\p)}-1$) and $\mu_{p^\infty}(H_{\p})$ is the (finite cyclic $p$-)group of $p$-th power roots of unity in $H_{\p}$.
\end{proposition}

\begin{proof} 
We have $$\A^*_H \cong (\R^*)^{r_1} \times (\C^*)^{r_2} \times \A^*_{H,f}$$ and $$\A^*_{H,f} \cong J_H \times \hat\cO_H^*.$$ Here, $J_H$ is the topologically discrete group of fractional ideals of $H$, so $J_H \cong \bigoplus_{\Z} \Z$, where the index runs over the set of prime ideals, and the entry of $\n \in J_H$ corresponding to a prime ideal $\p$  is given by $\mathrm{ord}_{\p}(\n)$. Furthermore,   $$\hat\cO_H^* = \prod_{\p \in M_{H,f}} \cO^*_{H,\p}$$ is the group of finite idelic units. To determine the isomorphism type of the latter, we quote \cite{Hasse}, Kapitel 15: let $\pi_{\p}$ be a local uniformizer at $\p$ and let $\bar{H}_{\p} = \cO_{H,\p}/\p$ denote the residue field; then the unit group is
\begin{equation} \label{unit} \cO^*_{H,\p} \cong \bar H_{\p}^* \times (1 + \pi_{\p}\cO_{H,\p}) \end{equation}
and the one-unit group \begin{equation} \label{oneunit} 1 + \pi_{\p}\cO_{H,\p} \cong \Z_p^{[H_{\p}:\Q_p]} \times \mu_{p^\infty}(H_{\p}) \end{equation} 
where  $\mu_{p^\infty}(H_{\p})$ is the group of $p$-th power roots of unity in $H_{\p}$. 
\end{proof}

 It remains to determine the exact structure of the $p$-th power roots of unity, e.g.:
 
 \begin{example}[\cite{AngSte}, Lemma 3.1 and Lemma 3.2] \label{m} If $H\neq \Q(i)$ and $H \neq \Q(\sqrt{-2})$, then there is an isomorphism of topological groups
$$\prod_{\p \in M_{H,f}} (\bar H_{\p}^* \times \mu_{p^\infty}(H_{\p}))  \cong  \prod_{n \geq 1} \Z/n\Z.$$ Hence we conclude: \emph{If $K$ and $L$  are two imaginary quadratic number fields different from $\Q(i)$ and $\Q(\sqrt{-2})$, then we have a topological group isomorphism $\A_K^* \cong \A_L^*$.}
\end{example} 

Combining Proposition \ref{a} and Example \ref{m}, we obtain the claim made in Example \ref{e1} in the introduction: 

\begin{corollary}
 For any two imaginary quadratic number fields $K$ and $L$ different from $\Q(i)$ and $\Q(\sqrt{-2})$ and for any integers $r$ and $s$, there are topological group isomorphisms
\[
(\A_{K,f})^r \times (\A_{K,f}^*)^s \cong (\A_{L,f})^r \times (\A_{L,f}^*)^s.
\]
and
$$ (\A_K)^r \times (\A_K^*)^s \cong (\A_L)^r \times (\A_L^*)^s. \qed $$
\end{corollary}

On the other hand, we again have a ``local'' result (and Remark \ref{remV} also applies in this case): 

\begin{proposition}
Let $K$ and $L$ be number fields such that there is a bijection $\varphi \colon M_{K,f} \rightarrow M_{L,f}$ with, for almost all places $\p$, isomorphisms of topological groups $\Phi_{\p} \colon (K^*_{\p},\cdot) \isomto (L^*_{\varphi(\p)},\cdot)$. Then $K$ and $L$ are arithmetically equivalent.
\end{proposition}

\begin{proof}
From (\ref{unit}) and (\ref{oneunit}), for a given $\p$ lying above $p\in \Q$, we find that
\[
K^*_{\p} \cong \Z \times \bar{K}^*_{\p} \times \Z_p^{[K_{\p}:\Q_p]} \times \mu_{p^\infty}(K_{\p}).
\]
Dividing out by the torsion elements yields
\[
K^*_{\p}/K^*_{\p,\mathrm{tors}} \cong \Z \times \Z_p^{[K_{\p}:\Q_p]}.
\]
Hence, $\Phi_{\p}$ will induce a map
\[
\Phi'_{\p} \colon K^*_{\p}/K^*_{\p,\mathrm{tors}} \isomto L^*_{\varphi(\p)}/L^*_{\varphi(\p),\mathrm{tors}}
\]
which equals
\[
\Phi'_{\p} \colon \Z \times \Z_p^{[K_{\p}:\Q_p]} \isomto \Z \times \Z_q^{[L_{\varphi(\p)}:\Q_q]}
\]
where $q \in \Q$ is the rational prime below $\varphi(\p)$. Now we form the quotient
\begin{equation}\label{p}
 (\Z \times \Z_p^{[K_{\p}:\Q_p]})/p\cdot(\Z \times \Z_p^{[K_{\p}:\Q_p]}) \cong \Z/p\Z \times (\Z/p\Z)^{[K_{\p}:\Q_p]}
\end{equation}
which, under $\Phi'_{\p}$, will map isomorphically onto
\begin{equation}\label{q}
(\Z \times \Z_q^{[L_{\varphi(\p)}:\Q_q]})/p\cdot(\Z \times \Z_q^{[L_{\varphi(\p)}:\Q_q]}) = 
\begin{cases}
\Z/p\Z \times (\Z/q\Z)^{[L_{\varphi(\p)}:\Q_q]} & \textrm{ if } p = q \\
  \Z/p\Z & \textrm{ if } p \neq q
\end{cases}.
\end{equation}
Thus, from comparing the right hand sides of equations (\ref{p}) and (\ref{q}), a counting argument shows that we must have $p=q$ and $[K_{\p}:\Q_p] = [L_{\varphi(\p)}:\Q_p]$. For all but finitely many primes, the extensions $K_{\p}/\Q_p$ and $L_{\varphi(\p)}/\Q_p$ are unramified.
Hence, we find that the bijection $\varphi$ matches the decomposition types of all but finitely many primes.  By Proposition \ref{Klingen}, this implies that $K$ and $L$ are arithmetically equivalent.
\end{proof}

\section{Set-up from algebraic groups and the notion of fertility} \label{nice}

In this section, we set up notations and terminology from the theory of algebraic groups, and we elaborate on the notion of a group being fertile for a pair of number fields. 

\subsection*{Algebraic groups}

\begin{nd} Let $G$ denote a 
linear (viz., affine) algebraic group. We denote the multiplicative group by $\G_m$ and the additive group by $\G_a$. An \emph{$n$-dimensional torus} $T$ is an algebraic subgroup of $G$ which is isomorphic, over $\bar \Q$, to $\G_m^r$, for some integer $r$. When $T$ is maximal, $r$ is the \emph{rank} of $T$ and $G$. Suppose that the maximal tori \emph{split} over a field $F/\Q$, meaning that there exists an isomorphism $T \cong \G_m^r$ defined over $F$. All split maximal tori of $G$ are $G(F)$-conjugate and have the same dimension, called the \emph{rank} of $G$. A subgroup $U$ of $G$ is \emph{unipotent} if $U(\bar \Q)$ consists of unipotent elements. Every unipotent subgroup of $G$ splits over $\Q$, meaning that it has a composition series over $\Q$ in which every successive quotient is isomorphic to $\G_a$. Alternatively, it is isomorphic over $\Q$ to a subgroup of a group of strictly upper triangular matrices. Any connected group $G$ that is not unipotent contains a non-trivial torus. A \emph{Borel subgroup} $B$ of $G$ is a maximal connected solvable subgroup of $G$. If all successive quotients in the composition series of $B$ over $F$ are isomorphic to $\G_a$ or $\G_m$, then $B$ is conjugate, over $F$, to a subgroup of an upper triangular matrix group, by the Lie-Kolchin theorem. Moreover, over $F$, for some split maximal torus $T$ and maximal unipotent group $U$, we can write $B \cong T \ltimes U$ as a semi-direct product induced by the adjoint representation $\rho \colon T \rightarrow \mathrm{Aut}(U)$ (i.e., by the conjugation action of $T$ on $U$). Furthermore, given $U$, $B$ is the normaliser of $U$ in $G$, and $T \cong B/U$. 
\end{nd} 

\begin{definition} \label{fertile}
We call  a linear algebraic group $G$ over~$\Q$ \emph{fertile} for a number field $K$ if the following conditions hold: 
\begin{enumerate} 
\item[(i)] The Borel groups $B=T \ltimes U$ of $G$ are split over $K$ (viz., $K \supseteq F$ where $F$ is the splitting field of a maximal torus in $G$);
\item[(ii)] over $K$, $T \neq\{1\}$ acts non-trivially (by conjugation) on the abelianisation $U^{\ab}$ of $U \neq \{0\}$. 
\end{enumerate}
\end{definition} 

The following equivalent definition was pointed out to us by Wilberd van der Kallen: 

\begin{proposition} \label{Wilb} $G$ is fertile over $K$ if and only if it splits over $K$ and the connected component of the identity $G^0$ is not a direct product $T \times U$ of a torus and a unipotent group.
\end{proposition} 

\begin{proof} 
 Indeed, suppose $G$ is fertile in the sense of Definition \ref{fertile}. Since Borel groups are connected, the identity component $G^0$ contains a Borel group, which is not a direct product $T \times U$, hence neither is $G^0$. 
 
 Conversely, suppose $G^0$ is not a direct product $T \times U$. There is a short exact sequence of algebraic groups
\[
1 \to R_u(G) \to G^0 \to S \to 1
\]
where $R_u(G)$ is the unipotent radical of $G$ and $S$ is a reductive group. Moreover, we may assume that $S$ is a torus. Otherwise, $S$ would contain a Borel group $\overline{B}$ whose abelian unipotent group $\overline{U}^{\mathrm{ab}}$ contains a non-trivial eigenspace for the maximal torus $\overline{T} \subset \overline{B}$. This eigenspace lifts to a non-trivial eigenspace inside $U^{\mathrm{ab}}$, so $G$ is fertile in the sense of Definition \ref{fertile}.

Thus, we may assume $G^0$ is itself a Borel subgroup (since it is solvable and connected, and maximal for these properties) with torus $S = T$ and unipotent subgroup $U = R_u(G)$, and consider the short exact sequence
\[
1 \to U \to B \to T \to 1.
\]
Now we claim that $T$ acts non-trivially on $U$ if and only if $T$ acts non-trivially on $U^{\mathrm{ab}}$, which will finish the proof.
Necessity is clear.
Conversely, suppose that $T$ acts non-trivially on $U$. Let $\{U_i\}_i$ be the lower central series for $U$, i.e. $U_0 = U$ and $U_i = [U,U_{i-1}]$ for all $i\geq 1$. Let $U_n$ be the last non-trivial subgroup occurring in this series. Then in particular $U_n \subseteq Z(U)$. Since $T$ acts by conjugation, it preserves the lower central series. We will show that $T$ acts non-trivially on $U^{\mathrm{ab}} = U/U_1$.\\
Since $U_k = \{0\}$ for $k > n$, we have $U/U_k = U$ for such $k$, on which $T$ acts non-trivially by assumption. We will show that if $T$ acts non-trivially on some $U/U_j$ with $j \geq 2$, then it acts non-trivially on $U/U_{j-1}$. This induction argument will then show that $T$ acts non-trivially on $U/U_1$, as required.

Suppose by contrapositive that $T$ acts trivially on some $U/U_{j-1}$. We will show that $T$ the also acts trivially on $U/U_{j}$. Indeed, $T$ acts trivially on the subgroup $U_{j-2}/U_{j-1} \leq U/U_{j-1}$. The commutator map $[\cdot,\cdot] \colon U \times U_{j-2} \to U_{j-1}$ factors through to give a map $$[\cdot,\cdot] \colon U/U_{j-1} \times U_{j-2}/U_{j-1} \to U_{j-1}/U_j.$$ Hence, $T$ acts trivially on $U_{j-1}/U_j$. Now consider the short exact sequence
\[
1 \to U_{j-1}/U_{j} \to U/U_{j} \to U/U_{j-1} \to 1.
\]
Since $T$ acts trivially on both $U_{j-1}/U_{j}$ and $U/U_{j-1}$, there are no non-trivial roots and $T$ acts trivially on $U/U_{j}$, as required. 
\end{proof}

\begin{examples} \mbox{ }
\begin{enumerate} 
\item[(i)] Tori and unipotent groups are \emph{not} fertile for any $K$, and neither are direct product of such groups. 

\item[(ii)] The general linear group $\GL(n)$ for $n \geq 2$ is fertile for any $K$. Here, $T$ is the group of diagonal matrices, split over $\Q$, which acts non-trivially on the group of strictly upper triangular matrices. Similarly, the Borel group of (non-strictly) upper triangular matrices is fertile. 

\item[(iii)] Let $G=\mathrm{Res}^F_{\Q} (\G_m \ltimes \G_a)$ denote the ``$ax+b$''-group of a number field $F$, as an algebraic group over $\Q$. This group is fertile for any number field $K$ that contains $F$. 
\end{enumerate} 
\end{examples}

\subsection*{Adelic point groups}

\begin{definition} \label{pointgroups} 
Let $G$ denote a linear algebraic group over $\Q$ and let $K$ a number field with adele ring $\A_K$.
As described in Section 3 of \cite{oesterle} (compare \cite{OC}) we may use any of the following equivalent definitions for the \emph{group of adelic points of $G$ over $K$} (also called the \emph{adelic point group}), denoted $G(\A_K)$:
\begin{enumerate}
\item Since $\A_K$ is a $\Q$-algebra, $G(\A_K)$ is its scheme theoretic set of points.
\item Let $S$ be a suitable finite set of places of $\Q$, including the archimedean place, and let $\mathcal{G}$ be a smooth separated group scheme of finite type over the $S$-integers $\Z_{S}$, whose generic fibre is $G$. Then 
\[ G(\A_K) = \lim_{\stackrel{\longrightarrow}{S' \supset S}} \prod_{\p \in S'} G(K_\p) \times \prod_{\p \notin S'} \mathcal{G}(\cO_\p) 
\] 
where  $S'$ runs over subsets of $M_{K,f}$ that contain divisors of primes in $S$. 

\item Choose a $\Q$-isomorphism $\varphi \colon G \hookrightarrow \mathbb{A}^N$ of $G$ onto a closed subvariety of a suitable affine space $\mathbb{A}^N$. For every $\p \in M_{K,f}$, we define $G(\mathcal{O}_{\p})$ to be the set of points $x \in G(K_{\p})$ for which $\varphi(x) \in \mathcal{O}_{\p}^N$. Then $G(\A_K)$ is the restricted product
\[
G(\A_K) = {\prod\limits_{\p \in M_{K,f}}}^{\hspace*{-2mm} \prime} \left(G(K_{\p}),G(\mathcal{O}_{\p})) \right) \times \prod_{\p \in M_{K,\infty}} G(K_{\p}) .
\]

\end{enumerate}
 The second and third definitions immediately provide $G(\A_K)$ with a topology induced from the $\p$-adic topologies. Also, the algebraic group law on $G$ induces a topological group structure on $G(\A_K)$. The definitions are, up to isomorphism, independent of the choices of $S$, $\mathcal{G}$ and $\varphi$. 
 
 We define the \emph{finite-adelic point group} $G(\A_{K,f})$ completely analogously.
\end{definition}

In Section \ref{counterex}, we considered the group of adelic points on $\G_a$ and $\G_m$, in the sense of the above definition, cf.\ Example \ref{e1} from the introduction.

\section{Divisibility and unipotency} \label{div}

In this section, we show how to characterise maximal unipotent point groups inside finite-adelic point groups in a purely group theoretic fashion, using divisibility. This is used later to deduce an isomorphism of unipotent point groups from an isomorphism of ambient point groups. 

\begin{definition} If $\mathbf H$ is a subgroup of a group $\mathbf G$, an element $h \in \mathbf H$ is called \emph{divisible in $\mathbf G$} if for every integer $n \in \Z_{>0}$, there exists an element $g \in \mathbf G$ such that $h=g^n$. The subgroup $\mathbf H \leq \mathbf G$ is called divisible (in $\mathbf G$) if all of its elements are divisible in $\mathbf G$. 
\end{definition}

\begin{proposition}\label{Urecovery} Fix a maximal unipotent algebraic subgroup $U$ of $G$. Then any maximal divisible subgroup of $G(\A_{K,f})$ is conjugate to $U(\A_{K,f})$ in $G(\A_{K,f})$. 
\end{proposition}

\begin{proof} We fix an embedding $G \hookrightarrow GL_N$ throughout, and consider elements of $G$ as matrices. We start the proof with a sequence of lemmas about the local case. Fix a place $\p \in M_{K,f}$. 

\begin{lemma} \label{div-uni}
All divisible elements of $G(K_\p)$ are unipotent.
\end{lemma}

\begin{proof}
Let $v$ denote a divisible element of $G(K_\p)$, and, for each $n \in \Z_{>0}$, let $w_n \in G(K_\p)$ satisfy $w_n^n=v$ for $n \in \Z_{>0}$. The splitting field $L_n$ of the characteristic polynomial of $w_n$ (seen as $N \times N$-matrix) has \emph{bounded} degree $[L_n:K] \leq N!$. Since by Krasner's Lemma (e.g., \cite{NSW}, 8.1.6), there are only finitely many extensions of $K_\p$ of bounded degree, the compositum $L$ of all $L_n$ is a discretely valued field, in which all the eigenvalues $\lambda_i$ of $v$ are $n$-th powers (namely, of eigenvalues of $w_n$) for all integers $n$.  Since $L$ is non-archimedean, $$\bigcap\limits_{n \geq 1} L^n = \{1\},$$ by discreteness of the absolute value and the structure of $\cO^*_{H}$ as described in the proof of Proposition \ref{mult}. We conclude that all eigenvalues of $v$ are $1$ and $v$ is unipotent. 
\end{proof}

\begin{remark}
The lemma (and hence the proposition) is not true for archimedean places. 
To give an example at a real place, the rotation group $\mathrm{SO}(2,\R) \subseteq \mathrm{SL}(2,\R)$ is divisible but contains non-unipotent elements. 
\end{remark} 

\begin{lemma} \label{uni-div}
The group $U(K_\p)$ is divisible in $G(K_\p)$. 
\end{lemma}

\begin{proof}
Since all $K_{\p}$ are fields of characteristic zero, the exponential map $$\exp \colon \mathfrak{N} \rightarrow U(K_\p)$$ from the nilpotent Lie algebra $\mathfrak{N}$ of $U(K_\p)$ to $U(K_\p)$  is an isomorphism (cf.\ \cite{milneLGS}, Theorem 6.5). For an integer $n \in \Z_{>0}$ and an element $\mathfrak{n} \in \mathfrak{N}$, $$\exp(n \mathfrak{n}) = \exp(\mathfrak{n})^n$$ by the Baker-Campbell-Hausdorff formula, since multiples of the same $\mathfrak n$ commute, so 
that (multiplicative) divisibility in the unipotent algebraic group corresponds to (additive) divisibility in the nilpotent Lie algebra $\mathfrak{N}$. Since the latter is a $K_\p$-vector space and any integer $n$ is invertible in $K_\p$, we find the result. 
\end{proof} 

We will also need the following global version: 

\begin{lemma} \label{div-global}
The group $U(\A_{K,f})$ is divisible in $G(\A_{K,f})$. 
\end{lemma}

\begin{proof}
Let $v = (v_{\p})_{\p} \in U(\A_{K,f})$, $n \in \Z_{\geq 0}$, and for every $\p \in M_{K,f}$, let $w_\p \in U(K_\p)$ be such that $w_\p^n=v_\p$ (which exists by the previous lemma). We claim that $w_\p \in U(\mathcal{O}_\p)$ for all but finitely many $\p$, which shows that $w=(w_\p)_{\p} \in U(\A_{K,f})$ and proves the lemma. Indeed, it suffices to prove that $w_\p \in GL_N(\cO_\p)$ for all but finitely many $\p$. This follows from the Taylor series 
$$ w_{\p} = \sqrt[n]{1+(v_\p-1)} = \sum_{k=0}^\infty \binom{1/n}{k}(v_\p-1)^k,$$
which is a finite sum since $v_\p-1$ is nilpotent, by noting that for fixed $n$, the binomial coefficients introduce denominators at only finitely many places.
\end{proof}

\begin{lemma} \label{conj}
Any maximal divisible subgroup of $G(K_\p)$ is conjugate to $U(K_\p)$ in $G(K_\p)$. 
\end{lemma}

\begin{proof}
Let $\D$ denote a maximal divisible subgroup of $G(K_\p)$. By Lemma \ref{div-uni}, it consists of unipotent elements. Since unipotency is defined by polynomial equations in the affine space of $N \times N$ matrices, the Zariski closure of $\D$ in $G_{K_\p}$ is a maximal unipotent algebraic subgroup $U'$ of $G_{K_\p}$.   Theorem 8.2 of Borel-Tits \cite{BorelTits} implies that there exists an element $\gamma_{\p} \in G(K_{\p})$ such that $\gamma_{\p}U'\gamma_{\p}^{-1} = U$, for $U$ any chosen maximal unipotent subgroup of $G$. 
Hence, $\gamma_{\p} \D \gamma_{\p}^{-1} \subseteq  U(K_{\p})$. On the other hand, Lemma \ref{uni-div} implies that $U(K_{\p})$ consists of divisible elements. Since $\gamma_{\p}^{-1} U(K_\p) \gamma_{\p}$ is maximal divisible, the result follows. 
\end{proof}

We could not find a proof for the following result in the literature, so we include one inspired by an answer by Bhargav Bhatt on MathOverflow {\tt mathoverflow.net/a/2231}: 

\begin{lemma} \label{decomp} 
Let $B\subset G$ denote a Borel subgroup of $G$ and let $\mathcal{B} \subset \mathcal {G}$ denote any corresponding inclusion of smooth finite-type separated group schemes over the ring of $S$-integers $\Z_S$ for a suitable finite set of primes $S$, so that the generic fibre of $\mathcal{B}$ is $B$ and that of $\mathcal {G}$ is $G$. Then for $\p \in M_{K,f}$ not dividing any prime in $S$, we have $$G(K_\p) = B(K_\p) \mathcal{G}(\cO_\p).$$ 
\end{lemma}

\begin{proof}
It suffices to show that \begin{equation} \label{quot} G(K_{\p})/B(K_{\p}) = \mathcal{G}(\mathcal{O}_{\p})/\mathcal{B}(\mathcal{O}_{\p}).\end{equation}
We will prove this by arguing that both sides of (\ref{quot}) equal $ (\mathcal{G}/\mathcal{B})(\mathcal{O}_{\p})$. 

First consider the long exact sequence in fppf-cohomology associated to the exact sequence $$0 \rightarrow \mathcal{B} \rightarrow \mathcal{G} \rightarrow \mathcal{G}/\mathcal{B} \rightarrow 0 $$ of smooth group schemes, over $M=K_\p$ or $M=\cO_\p$: 
\[
0 \to \mathcal{B}(M) \to \mathcal{G}(M) \to (\mathcal{G}/\mathcal{B})(M) \to H^1(\mathrm{Spec\, M},\mathcal{B}) \to \ldots
\]
To rewrite the left hand side of (\ref{quot}), we take $M=K_\p$, so we are dealing with $\mathrm{Gal}(\bar K_\p/K_\p)$-cohomology. 
Since $B$ is a connected linear algebraic group, $H^1(K_{\p},B) = 0$ (cf.\ Theorem 1 of Section III.2.2 of \cite{serre}). Hence, $$G(K_{\p})/B(K_{\p})=(G/B)(K_{\p}).$$ Since $G/B$ is a projective variety, we know that $$(G/B)(K_{\p}) = (\mathcal{G}/\mathcal{B})(\mathcal{O}_{\p}).$$

For the right hand side of (\ref{quot}), we set $M=\cO_\p$ and argue as in Step 3 of Theorem 6.5.12 of \cite{poonen}: $H^1(\mathcal{O}_{\p},\mathcal{B})$ classifies $\mathcal{B}$-torsors over $\cO_\p$; let $\mathcal T \rightarrow \mathrm{Spec}\, \cO_\p$ denote such a torsor. By Lang's theorem, its special fibre $\mathcal{T}_\p \rightarrow \mathrm{Spec}\, \mathbf{F}_{\p}$ over the finite residue field $\mathbf{F}_\p$ has a rational point. Since $\mathcal{B}$ smooth, so is $\mathcal{T}$, so we can lift the rational point by Hensel's Lemma. Hence, $\mathcal T$ is also trivial. We conclude that $H^1(\mathrm{Spec}\, \cO_\p,\mathcal B)=0$, so $$(\mathcal{G}/\mathcal{B})(\mathcal{O}_{\p}) =  \mathcal{G}(\mathcal{O}_{\p})/\mathcal{B}(\mathcal{O}_{\p})$$ as claimed.
\end{proof} 

To finish the proof of the proposition, let $\D$ denote a maximal divisible subgroup of $G(\A_{K,f})$ and let $\D_\p = \D \cap G(K_\p)$ be its local component for $\p \in M_{K,f}$. Let $\gamma_\p \in G(K_\p)$ be as in Lemma \ref{decomp}, i.e., such that $\gamma_\p \D_\p \gamma_p^{-1}  = U(K_\p)$. Since $B=N_G(U)$ is the normaliser of $U$, Lemma \ref{decomp} implies that for all but finitely many $\p$, we may replace $\gamma_\p$ by an element in $\mathcal{G}(\cO_\p)$, which we again denote by $\gamma_\p$ for ease of notation. In this way, we find $\gamma = \prod\limits_{\p \in M_{K,f}} \gamma_\p \in G(\A_{K,f})$ with 
\[
\D \subseteq \gamma^{-1} \prod U(K_\p) \gamma \cap G(\A_{K,f}) = \gamma^{-1} U(\A_{K,f}) \gamma \subseteq \D,
\]
where the last inclusion holds by Lemma \ref{div-global}. \end{proof} 

\section{Proof of Theorem \ref{mainG}} \label{red3}

We now turn to the proof of Theorem \ref{mainG}. We first recall its statement:
\begin{theorem}[Theorem \ref{mainG}]
Let $K$ and $L$ be two number fields, and let $G$ denote a 
linear algebraic group over $\Q$, fertile for $K$ and $L$. There is a topological group isomorphism of adelic point groups $G(\A_{K,f}) \cong G(\A_{L,f})$ if and only if there is a topological ring isomorphism $\A_K \cong \A_L$. 
\end{theorem}

\begin{proof} 
Let $\mathbf G := G(\A_{K,f})$ as a topological group. We will apply purely group theoretic construction to $\mathbf G$, to end up with the adele ring $\A_K$; this shows that the isomorphism type of the adele ring is determined by the topological group $\mathbf G$. Let $\D$ denote a maximal divisible subgroup of $\mathbf G$. Consider the normaliser $\mathbf N:= N_{\mathbf G} \D$ of $\D$ in $\mathbf G$. Let $\mathbf V := [\mathbf N, \D]/[\D,\D] \leq \D^{\ab}$, and let $\mathbf T := \mathbf N / \mathbf D$. Note that $\mathbf T$ acts naturally on $\mathbf V$ by conjugation. 
Since $\mathbf V$ is locally compact Hausdorff, we can give $\mathrm{End} \mathbf V$, the endomorphisms of the abelian group $\mathbf V$, the compact-open topology.

\begin{proposition} \label{end}
There exists an integer $\ell \geq 1$ such that there is a topological ring isomorphism $$Z(\mathrm{End}_{\mathbf T} \mathbf V) \cong \A_{K,f}^{\ell},$$ where the left hand side is the centre of the ring of continuous endomorphism of the $\mathbf T$-module $\mathbf V$. 
\end{proposition}

\begin{proof}[Proof of Proposition \ref{end}] First, we relate the subgroups of $\mathbf G$ to points groups of algebraic subgroups of $G$. From Proposition \ref{Urecovery}, we may assume that $\D = U(\A_{K,f})$ for a maximal unipotent algebraic subgroup of $G$. 
The normaliser of $U$ in $G$ as an algebraic group is a Borel group $B$ inside $G$ (Theorem of Chevalley, e.g.\ \cite{MR1102012}, 11.16). Since taking points and taking normalisers commute (\cite{milneAGS}, Proposition 6.3), we obtain that 
\[
\mathbf N = N_{\mathbf G} \D  = N_{G(\A_{K,f})}U(\A_{K,f}) = (N_G U)(\A_{K,f}) = B(\A_{K,f}).
\]
and $\mathbf T \cong T(A_{K,f})$ for $T$ any maximal torus in $B$. 

Next, we analyse the action of $\mathbf T$ on $\mathbf V$, knowing the action of $T$ on $U$. The hypothesis that $T$ splits over $K$  implies that $T \cong \G_m^{r}$ over $K$ for some $r$. The adjoint action of $T$ by conjugation on $U$ maps commutators to commutators, so it factors through the abelianisation $U^{\ab}$, and we can consider the linear adjoint action $\rho \colon T \rightarrow \mathrm{Aut}(U^{\ab})$ over $K$. Note that $U^{\ab} \cong \G_a^k$ for some integer $k$, so we have an action over $K$
\begin{equation}
\rho \colon T (\cong \G_m^{r}) \to \mathrm{Aut}(\G_a^k)=\GL(k),
\end{equation}
which  is diagonalisable over $K$ as a direct sum $\rho = \oplus \chi_i$ of $k$ characters $\chi_i \in \Hom_K(T,\G_m)$ of algebraic groups. In coordinates $t~=~(t_1, \ldots, t_{r}) \in \G_m^r = T$, any such character is of the form \begin{equation} \label{ch} \chi (t) = \chi(t_1, \ldots, t_{\ell})= t_1^{n_1} \cdot \ldots \cdot t_{r}^{n_{r}} \end{equation} 
for some $n_1, \ldots, n_{r} \in \Z$.   
Since the action of $\mathbf T$ on $\mathbf V$ is given by specialisation from the action of $T$ on a subspace of $U^{\ab}$, we find an isomorphism of $\mathbf T$-modules $$\mathbf V \cong  \bigoplus_{i=1}^{\ell} \A_{K,f,\chi_i}^{\mu_i},$$ where  $\chi_i$ ($i=1,\dots \ell$) are the distinct non-trivial characters that occur in $\mathbf V$, $\mu_i$ is the multiplicity of $\chi_i$ in $\mathbf V$, and $\A_{K,f,\chi_i}$ is the $\mathbf T$-module $\A_{K,f}$ where $\mathbf T$ acts via $\chi_i$. 
Hence, 
\begin{equation} \label{mat} \mathrm{End}_{\mathbf T} \mathbf V = \prod_{i=1}^\ell \prod_{j=1}^\ell \mathrm{Mat}_{\mu_j\times\mu_i} \left( \mathrm{Hom}_{\mathbf T}(\A_{K,f,\chi_i},\A_{K,f,\chi_j}) \right). \end{equation} 

The assumption of fertility means precisely that $\ell \geq 1$. 

\begin{lemma} \label{si}
If $\chi_i$ and $\chi_j$ are non-trivial characters occurring in the above decomposition, then there is a topological ring isomorphism $$\mathrm{Hom}_{\mathbf T}(\A_{K,f,\chi_i},\A_{K,f,\chi_j}) \cong \begin{cases} \A_{K,f} &\mbox{if } \chi_i = \chi_j \\
\{0\} & \mbox{otherwise }. \end{cases}$$
\end{lemma}

\begin{proof}
A homomorphism of additive groups $f \colon (\A_{K,f,\chi_i},+) \rightarrow (\A_{K,f,\chi_j},+)$ is $\T$-equivariant precisely if $f(\chi_i(t)(u))=\chi_j(t)f(u)$ for all $t \in \mathbf T$ and $u \in \A_{K,f}$. The $\chi$ are specialisations of algebraic characters as in (\ref{ch}), and some powers are non-zero by the assumption of fertility. If $\chi_i \neq \chi_j$, this means that 
\begin{equation} f(t^nu)=t^mf(u), \forall t \in \A_{K,f}^*, \ \forall u \in \A_{K,f} 
\end{equation}
for some $n,m>0$, $n \neq m$, which is impossible unless $f=0$: indeed, choose $u$ such that $f(u)$ is not a zero divisor, and choose $t \in \Z_{>0}$; then the equation says that $t^n f(u) = t^m f(u)$, so $m=n$. So we must have $\chi_i = \chi_j$, and we find that 
\begin{equation} \label{f} f(t^nu)=t^nf(u), \forall t \in \A_{K,f}^*, \ \forall u \in \A_{K,f} 
\end{equation}
for some $n>0$. 

We now reinterpret a formula of Siegel (\cite{MR0009778}, p.\ 134) as saying the following: \emph{
Let $R$ denote a ring and $n$ a positive integer such that $n!$ is invertible in $R$. Then any element of $R$ belongs to the $\Z$-linear span of the $n$-th powers in $R$. In particular, we have the following explicit formula for any $z \in R$: 
\begin{equation*}
z = \sum_{k=0}^{n-1} (-1)^{n-k-1} \binom{n-1}{k} \left\{ \left( \frac{z}{n!}+k\right)^n-k^n \right\}. 
\end{equation*}
}
Applied to $R=\A_{K,f}$, in which $n!$ is invertible, Siegel's formula expresses any element of $\A_{K,f}$ as $\Z$-linear combination of $n$-th powers in $\A_{K,f}$. We conclude from 
(\ref{f}) and additivity of $f$ that 
\begin{equation} f(tu)=tf(u), \forall t \in \A_{K,f}^*, \ \forall u \in \A_{K,f}. \end{equation}
Hence, $f(t)=tf(1)$ is completely determined by specifying a value for $f(1) \in \A_{K,f}$, and $$\mathrm{End}_{\mathbf T} (\A_{K,f,\chi}) \rightarrow \A_{K,f} \colon f \mapsto f(1)$$ is the required ring isomorphism.  It is continuous, since evaluation maps (such as this one) are continuous in the compact-open topology on $\mathrm{End}_{\mathbf T} (\A_{K,f,\chi})$. The inverse map is $\alpha \mapsto (x \mapsto \alpha x)$, which is also continuous in the finite-adelic topology on $\A_{K,f}$. Hence, we find an isomorphism of topological groups, as required.
\end{proof}
To finish the proof of Proposition \ref{end}, combine Lemma \ref{si} with Equation \ref{mat} and takes centres: 
 \begin{equation*} Z\left(\mathrm{End}_{\mathbf T} \mathbf V\right) = Z\left(\prod_{i=0}^\ell M_{\mu_i} \left( \A_{K,f} \right) \right) = \A_{K,f}^{\ell}. \end{equation*} 
\end{proof}

If $R$ is a ring, let $\cM(R)$ denote its set of principal maximal ideals. Observe that $\cM(R^\ell) = \cM(R) \times \Z/\ell\Z$, since a maximal ideal in $R^\ell$ is of the form $R^{\ell_1} \times \m \times R^{\ell_2}$ for some maximal ideal $\m$ of $R$ and a decomposition $\ell = \ell_1+\ell_2+1$. Now we recall the description of the principal maximal ideals in an adele ring $\A_{K,f}$ as given by Iwasawa and Lochter (\cite{Lochter}, Satz 8.6 and \cite{Iwasawa}, p.\ 340--342, cf.\ \cite{MR1638821}, VI.2.4): 
$$\cM(\A_{K,f}) = \{  \m_{\p} = \ker \left(\A_{K,f} \rightarrow K_{\p}\right)  \}. $$
Note that $\A_{K,f}/\m_{\p} \cong K_{\p}$. 
Hence the multiset 
$$ \{ \A^\ell_{K,f}/\m \colon \m \in \cM(\A^\ell_{K,f}) \} $$
contains a copy of the local field $K_\p$ exactly $\ell r_\p$ times, where $r_\p$ is the number of local fields of $K$ isomorphic to $K_\p$.
Thus, we have constructed the multiset of local fields
$$ \{ K_{\p} \colon \p \in M_{K,f} \}$$ of $K$, up to isomorphism of local fields.

Now if $K$ and $L$ are two number fields with $G(\A_{K,f}) \cong G(\A_{L,f})$ as topological groups, then these multisets are in bijection, i.e., there exists a bijection of places $\varphi \colon M_{K,f} \rightarrow M_{L,f}$ such that $K_{\p} \cong L_{\varphi(\p)}$ for all $\p \in M_{K,f}$. 
Hence $K$ and $L$ are locally isomorphic (in the sense of Section \ref{deflocal}), and we find ring isomorphisms $\A_{K,f} \cong \A_{L,f}$ and $\A_K \cong \A_L$, by Proposition \ref{Klingen}.  

For the reverse implication $\A_K \cong \A_L \Rightarrow G(\A_{K,f}) \cong G(\A_{L,f})$, we use that a topological ring isomorphism $\A_K \cong \A_L$ implies the existence of topological isomorphisms $\Phi_\p \colon K_\p \cong L_{\varphi(\p)}$ of local fields for some bijection of places $\varphi \colon M_{K,f} \rightarrow M_{L,f}$ (again by Proposition \ref{Klingen}). The fact that all $\Phi_\p$ are homeomorphisms implies in particular that $\Phi_\p(\cO_{K,\p}) = \cO_{L,\varphi(\p)}$ for all $\p$. Now $G(\A_{K,f}) \cong G(\A_{L,f})$ is immediate from the definition of finite-adelic point groups (with topology) in \ref{pointgroups}(2) or, equivalently, \ref{pointgroups}(3). This finishes the proof of Theorem \ref{mainG}. 
\end{proof}

\begin{remark}
The proof also shows that if $G$ is fertile for two number fields $K$ and $L$, then there is an abstract group isomorphism of adelic point groups $G(\A_{K,f}) \cong G(\A_{L,f})$ if and only if there is an abstract ring isomorphism $\A_K \cong \A_L$. 
\end{remark}

\section{Hecke algebras and proof of Theorem \ref{mainH}} \label{HH}

\begin{definition}\label{Heckealgebra} Let $G$ denote a linear algebraic group over $\Q$. Then $\G_K:=G(\A_{K,f})$ is a locally compact topological group (since $\A_{K,f}$ is locally compact) for the topology described in Definition \ref{pointgroups}, equipped with a (left) invariant Haar measure $\mu_{\G_K}$. The finite (or non-archimedean) real \emph{Hecke algebra} $\cH_G(K)=C^{\infty}_c(\G_{K},\R)$ of $G$ over $K$ is the algebra of all real-valued locally constant compactly supported continuous functions $\Phi: \G_K \rightarrow \R$ with the convolution product:
\[
 \Phi_1 \ast \Phi_2 : g \mapsto \int_{\G_K} \Phi_1(gh^{-1})\Phi_2(h)d\mu_{\G_K}(h).
\]
\end{definition}
Every element of $\cH_G(K)$ is a finite linear combination of characteristic functions on double cosets $\mathbf{K}h\mathbf{K}$, for $h \in \G_K$ and $\mathbf{K}$ a compact open subgroup of $\G_K$. Alternatively, we may write 
$$ \cH_G(K) = \lim_{\substack{\longrightarrow \\ \mathbf{K}}} \cH(\G_K /\!\!/ \mathbf{K}), $$
where $\cH(\G_K/\!\!/\mathbf{K})$ is the Hecke algebra of $\mathbf{K}$-biinvariant smooth functions on $\G_K$ (for example, if $\mathbf{K}$ is maximally compact, this is the spherical Hecke algebra).

\begin{definition} If $\G$ is a locally compact topological group equipped with a Haar measure $\mu_{\G}$, let $L^1(\G)$ denote its \emph{group algebra}, i.e., the algebra of real-valued $L^1$-functions with respect to the Haar measure $\mu_{\G}$, under convolution. 

An isomorphism of Hecke algebras $ \Psi \colon \cH_G(K) \isomto \cH_G(L)$ which is an isometry for the $L^1$-norms arising from the Haar measures (i.e., $||\Psi(f)||_1=||f||_1$ for all $f \in \cH_G(K)$) is called an \emph{$L^1$-isomorphism.} \end{definition}

Before we give its proof, let us recall the statement of Theorem \ref{mainH}:
\begin{theorem}[Theorem \ref{mainH}]
Let $K$ and $L$ be two number fields, and let $G$ denote a
linear algebraic group over $\Q$, fertile for $K$ and $L$. There is an $L^1$-isomorphism of Hecke algebras $\cH_G(K) \cong \cH_G(L)$ if and only if there is a ring isomorphism $\A_K \cong \A_L$. 
\end{theorem}

\begin{proof}
The proof consists of two steps: first we show, using the Stone-Weierstrass theorem, that the Hecke algebras are dense in the group algebras, and then we use results on reconstructing a locally compact group from its group algebra due to Kawada (for positive maps) and Wendel (for $L^1$-isometries).  
\subsubsection*{Step 1: An $L^1$-isomorphism $\cH_G(K) \cong \cH_G(L)$ implies an $L^1$-isomorphism $L^1(\G_K) \cong L^1(\G_L)$. 
}\label{o1} 
The locally compact real version of the Stone-Weierstrass theorem implies that $\cH_G(K)$ is dense in $C_0(\G_K)$ for the sup-norm, where $C_0(\G_K)$ denotes the functions that vanish at infinity, i.e., such that $|f(x)|< \varepsilon$ outside a compact subset of $\G_K$. Indeed, one needs to check the nowhere vanishing and point separation properties of the algebra (\cite{HS}, 7.37.b). Since $\cH_G(K)$ contains the characteristic function of any compact subset $\mathbf{K} \subseteq \G_K$, the algebra vanishes nowhere, and the point separating property follows since $\G_K$ is Hausdorff. 

A fortiori, $\cH_G(K)$ is dense in the compactly supported functions $C_c(\G_K)$ for the sup-norm, and hence also in the $L^1$-norm. Now $C_c(\G_K)$ is dense in $L^1(\G_K)$, and the claim follows. 

\subsubsection*{Step 2:   
An $L^1$-isometry $L^1(\G_K) \cong L^1(\G_L)$ implies a group isomorphism $\G_K \cong \G_L$}
Indeed, 
an $L^1$-isometry  $\cH_G(K) \cong \cH_G(L)$ implies an $L^1$-isometry of group algebras $L^1(\G_K) \cong L^1(\G_L)$. Hence the result follows from combining theorems of Wendel \cite{MR0049910} and Kawada \cite{Kawada}, which prove that an $L^1$-isometry of group algebras of locally compact topological groups is always induced by an isomorphism of the topological groups.

Finally, Theorem \ref{mainG} says that if $G$ is fertile, then $\G_K \cong \G_L$ implies that $\A_{K} \cong \A_{L}$. 
\end{proof}

\begin{corollary} \label{fh}
If $G$ is a
connected linear algebraic group over $\Q$ which is fertile for $K$ and $L$, where $K$ and $L$ are two numbers fields which are Galois over $\Q$, then an $L^1$-isomorphism of Hecke algebras $\cH_G(K) \cong \cH_G(L)$ implies that the fields $K$ and $L$ are isomorphic.
\end{corollary}

\begin{proof}
Since the hypotheses imply that the fields are arithmetically equivalent, the result follows from Proposition \ref{Klingen}.(iii). 
\end{proof} 

Since $\GL(n)$ is fertile if $n \geq 2$, we obtain Corollary \ref{corc}. 

\subsection*{Variations on Theorem \ref{mainH}} 
\begin{enumerate} 
\item The theorem is also true if the real-valued Hecke algebra is replaced by the complex-valued Hecke algebra (using the complex versions of Stone-Weierstrass and Kawada/Wendel). 
\item It seems that the theorem also holds for the full Hecke algebra $\cH_G \otimes\cH_G^{\infty},$ where $\cH_G^\infty$ is the archimedean Hecke algebra for $G$, viz., the convolution algebra of distributions on $G(\R\otimes_{\Q} K)$ supported on a maximal compact subgroup of $G(\R\otimes_{\Q} K)$, 
but we have not checked the analytic details. 
\end{enumerate} 

\section{Discussion}

One may wonder in what exact generality Theorem \ref{mainG} holds.
\begin{enumerate} 
\item The theorem does not hold for all linear algebraic groups; see, e.g., Example \ref{e1}. Is it possible to characterise \emph{precisely} the linear algebraic groups for which $G(\A_K) \cong G(\A_L)$ implies $\A_K \cong \A_L$? 
\item What happens if $G$ is not a linear algebraic group, but any algebraic group? It follows from Chevalley's structure theorem that such $G$ have a unique maximal linear subgroup $H$; can we deduce $H(\A_K) \cong H(\A_L)$ from $G(\A_K) \cong G(\A_L)$?
\item What happens if there is no linear part, i.e., $G$ is an abelian variety, e.g., an elliptic curve? For every number field, is there a sufficiently interesting elliptic curve $E/\Q$ such that $E(\A_K)$ determines all localisations of $K$? 
\item Is the theorem true without imposing that the maximal torus $T$ of $G$ splits over $K$ and $L$?
\item What happens over global fields of positive characteristic?  
\item The category of $\cH_G(K)$-modules does not seem to determine the field $K$ (cf.\ forthcoming work of the second author). Can the category be enriched in some way so as to determine $K$? Our theorem suggests to try and keep track of some analytic information about $\cH_G(K)$ related to the $L^1$-norm. 
\end{enumerate}

\end{document}